\documentclass[letterpaper,12pt,reqno]{amsart}
\usepackage[margin=1in]{geometry}

\usepackage{amsmath}
  \usepackage{paralist}
 
\usepackage{graphicx}  
\usepackage{amssymb,amsthm,paralist,color,tikz-cd}
\usepackage{mathrsfs}
\usepackage{mathtools}
\usepackage{cite}
\usepackage{setspace}
\usepackage{amscd}
\usepackage{multicol}
\usepackage{float}
 \usepackage[colorlinks=true]{hyperref}
   
\hypersetup{urlcolor=blue, citecolor=red}

\newtheorem{theorem}{Theorem}[section]
\newtheorem{corollary}{Corollary}

\newtheorem{lemma}[theorem]{Lemma}
\newtheorem{proposition}{Proposition}

\theoremstyle{definition}
\newtheorem{definition}[theorem]{Definition}
\newtheorem{remark}{Remark}

\newcommand{\N}{\mathbb{N}}
\newcommand{\Z}{\mathbb{Z}}
\newcommand{\Q}{\mathbb{Q}}
\newcommand{\R}{\mathbb{R}}
\newcommand{\C}{\mathbb{C}}
\renewcommand{\H}{\mathbb{H}}
\newcommand{\mrm}{\mathrm}

\newcommand{\mc}{\mathcal}

\renewcommand{\a}{\alpha}
\renewcommand{\b}{\beta}
\newcommand{\g}{\gamma}
\newcommand{\G}{\Gamma}
\renewcommand{\d}{\delta}
\newcommand{\e}{\varepsilon}
\renewcommand{\l}{\lambda}
\renewcommand{\L}{\Lambda}

\newcommand{\s}{\sigma}
\newcommand{\vp}{\varphi}
\renewcommand{\t}{\tau}
\renewcommand{\th}{\theta}

\newcommand{\set}[1]{\left\{#1\right\}}
\renewcommand{\r}{\rightarrow}

\def\multiset#1#2{\ensuremath{\left(\kern-.3em\left(\genfrac{}{}{0pt}{}{#1}{#2}\right)\kern-.3em\right)}}

\title[Recurrence and Rigidity]{Geodesic Planes in Geometrically Finite Manifolds}


\author[Osama Khalil]{Osama Khalil}
\address{Department of Mathematics, Ohio State University, Columbus, OH}
\email{khalil.37@osu.edu}

\begin{document}

\begin{abstract}
We study the problem of rigidity of closures of totally geodesic plane immersions in geometrically finite manifolds containing rank $1$ cusps. We show that the key notion of $K$-thick recurrence of horocycles fails generically in this setting. This property played a key role in the recent breakthroughs of McMullen, Mohammadi and Oh. Nonetheless, in the setting of geometrically finite groups whose limit sets are circle packings, we derive $2$ density criteria for non-closed geodesic plane immersions, and show that closed immersions give rise to surfaces with finitely generated fundamental groups. We also obtain results on the existence and isolation of proper closed immersions of elementary surfaces.
\end{abstract}

\maketitle

\section{Introduction}

\subsection{Formulation of Results}
	
	Let $M$ be a hyperbolic manifold of dimension $3$. Let
    $f: \H^2 \rightarrow M$
    be a totally geodesic immersion of the hyperbolic plane.
    Ratner~\cite{Ratner1991} and Shah~\cite{Shah} independently classified the possibilities for the closures of $f(\H^2)$ inside $M$, in the case when $M$ has finite volume.
    Recently, in~\cite{MMO-Planes}, a complete classification was obtained in the case when $M$ is a convex cocompact manifold whose convex core has a totally geodesic boundary
    \footnote{In a more recent preprint, the authors also obtained a nearly full classification in the case of  general acylindrical convex cocompact manifolds.}.
    This class of manifolds will be referred to as rigid acylidrical manifolds.
    In both cases, it is shown that such closures are always immersed submanifolds of $M$.
    See also the related results in~\cite{MMO-horocycles,MaucourantSchapira} on the topological dynamics of horocycles in $3$ manifolds.

    In this article, we study the rigidity problem of closures of geodesic plane immersions in geometrically finite manifolds containing rank-$1$ cusps.
    We begin by showing that a certain desirable recurrence property of unipotent orbits ($K$-thickness) fails generically in this setting, Theorem~\ref{almost every U orbit is non-recurrent}.
    This property plays a crucial role in~\cite{MMO-Planes}.
    We then formulate sufficient conditions for certain geodesic plane immersions to be dense, Theorems~\ref{K thick implies plane denseness} and~\ref{accumulating on a horocycle implies dense}.
    In addition, when the limit set associated with our kleinian manifold $M$ is a circle packing, we show that a closed geodesic plane immersion gives rise to a geometrically finite surface (i.e. having a finitely generated fundamental group) and further characterize the limit set of such a surface, Theorem~\ref{limit set of a closed orbit}.
    Finally, we establish the existence of properly immersed elementary surfaces (Theorem~\ref{finite intersection implies closed}) and study the situation when a compact set meets infinitely many of them (Proposition~\ref{finiteness of orbits in each degree}).

\subsection{Preliminary Notions}  
    To state our results precisely, we recall some necessary notions and refer the reader to \S~\ref{section: prelims} for detailed definitions.
    
    The frame bundle of $\H^3$, denoted by $\mrm{F}\H^3$, consists of the set of orthonormal $3$-frames at every point in $\H^3$.
    The group $G=\mrm{Isom}^+(\H^3)\cong \mrm{PSL}_2(\C)$ acts simply transitively on $\mrm{F}\H^3$ and the two can thus be identified.
    The following $1$-parameter subgroups of $G$ will be important to us.
    \begin{align*}
    	U = \set{u_t =\begin{pmatrix} 1 & t\\ 0 & 1 \end{pmatrix}: t\in \R},
        A = \set{a_t=\begin{pmatrix} e^{t/2} & 0\\ 0 & e^{-t/2} \end{pmatrix}: t\in \R}.
    \end{align*}
    The action of the group $A$ induces the frame flow and generates the parametrized geodesics on the unit tangent bundle $\mrm{T}^1\H^3$.
    Orbits of the group $U$ project to horocycles in $\H^3$.
    
    Let $\G$ be a Kleinian group (a discrete subgroup of $G$) so that $M$ is identified with $\G\backslash \H^3$.
    The limit set $\L$ of $\G$ is defined to be
    \[ \L = \overline{\G \cdot o} \cap \partial \H^3, \]
    where $o\in \H^3$ is any point and $\partial \H^3$ is the boundary at infinity.
    In the ball model of $\H^3$, we can identify $\partial \H^3$ with $\mathbb{S}^2$.
    The non-wandering set for the $A$ action on $\G \backslash \mrm{F} \H^3$ consists of those frames $(v_1,v_2,v_3)$
    for which any lift of the geodesic tangent to $v_1$ to $\H^3$ joins two points on the boundary belonging to $\L$.
    Keeping the notation of~\cite{MMO-Planes}, we denote this set by $\mrm{RFM}$.
    For a frame $x\in \mrm{F}\H^3$, we will use $x^+$ (resp. $x^-$) to denote the forward (resp. backward) endpoint
    in $\partial \H^3$ of the geodesic tangent to its first vector.
    
    By lifting a geodesic immersion of $\H^2$ to the universal cover, we obtain a totally geodesic hyperplane inside $\H^3$.
    The closure of any such geodesic hyperplane inside $\H^3 \cup \partial \H^3$ meets the boundary in a Euclidean circle.
    We will denote the space of circles in $\partial \H^3$ by $\mc{C}$ and the subset of $\mc{C}$ which meets $\L$ by $\mc{C}_\L$.
    
    The group $G$ acts transitively on $\mc{C}$ and the stabilizer of a point is a conjugate of the subgroup
    $H = \mrm{PGL}_2(\R)$.
    Hence, we can identify $\mc{C}$ with $G/H$.
    Orbits of the right action of $H$ on $\mrm{F}\H^3$ project to all geodesic hyperplanes in $\H^3$.
    Thus, understanding the orbit closures of $H$ on $\G\backslash \mrm{F}\H^3$ is a finer question than the closures of totally geodesic immersions on $\G \backslash \H^3$.
    
    The action of $\G$ on $\H^3$ extends to a conformal action on $\partial \H^3$.
    In particular, circles are mapped to circles, extending this action to an action on $\mc{C}$ which agrees with the action of $\G$ on $G/H$ by left multiplication.
    Orbit closures of $\G$ acting on $\mc{C} = G/H$ are in one-to-one correspondence with orbit closures of $H$ on $\G\backslash G$.
    Following~\cite{MMO-Planes}, we shall adopt this point of view throughout the paper.

\subsection{Dynamics of Horocycles}
    The key technique used in the classification problem in finite volume dates back to Margulis' resolution of the Oppenheim conjecture~\cite{MargulisOppenheim}.
    It relies on the fact that unipotent trajectories spend a positive proportion of their time in big compact sets.
    This fails in infinite volume but a suitable substitute was introduced in~\cite{MMO-Planes}.
	A set $T\subseteq \R$ is said to be $\mathbf{K}$-\textbf{thick} for some $K>1$ if for all $t>0$, 
	\[ T \cap \left([-Kt,-t] \cup [t,Kt] \right) \neq \emptyset. \]
	    
    In the context of rigid acylindrical manifolds, it was shown in~\cite{MMO-Planes} that the set of return times to $\mrm{RFM}$, of the unipotent orbit of a frame $x\in \mrm{RFM}$ is $K$-thick, for some $K>1$ which is independent of $x$.
    This key property was shown to be sufficient to carry out the classical polynomial shearing arguments of unipotent trajectories which drive transverse smoothness of the closures of totally geodesic planes.

    Our first result says that the presence of rank-$1$ parabolics prohibits even such weak notion of recurrence.
    The following is the precise statement.
	\begin{theorem} \label{almost every U orbit is non-recurrent}
		Assume $\G$ is a geometrically finite Kleinian group containing rank-$1$ parabolic subgroups. Suppose $x\in \mrm{RFM}$ is such that $\overline{xA} = \mrm{RFM}$. Then, the set of return times
        \[ R(x) = \set{t\in \R: xu_t \in \mrm{RFM}} \]
        is not $K$-thick for any $K>1$.
	\end{theorem}

    \subsection{Density Criteria}
    The lack of recurrence of unipotent orbits significantly limits the applicability of many classical techniques to the rigidity problem.
    Thus, we restrict our attention to geometrically finite manifolds whose limit sets on the boundary of $\H^3$ form a circle packing in order to leverage their geometric structure in our study.
    Recall that the limit set is said to be a \textbf{circle packing} if it is the complement of countably many open round disks in $\partial \H^3$.

    An important example of circle packings which arise as limit sets of geometrically finite Kleinian groups is Apollonian circle packings, see figure~\ref{fig:apollonian}.
    In these examples, stabilizers of tangency points between circles in the packing give rise to rank $1$ cusps.
    Geometric and number theoretic properties of (Apollonian) circle packings have been studied extensively in the literature. See~\cite{OhShah,GrahamLagarias} for example.
    Beyond the Apollonian group, it is shown in~\cite{Keen93geometricfiniteness} that every isomorphism class of geometrically finite groups contains one whose limit set is a circle packing.
    
    In this setting, we prove two criteria for density of totally geodesic immersions which we now discuss.
	Our first rigidity criterion demonstrates that $K$-thick recurrence of certain unipotent orbits can in fact be sufficient for the geodesic plane containing these horocycles to be dense. 
    
    We say that a point $\a\in\L$ is \textbf{rank-}$\mathbf{1}$ \textbf{unbounded} if there exists a rank-$1$ parabolic fixed point $\s \in \L$ such that for every horoball $H$ centered at $\s$ and every geodesic ray $\g:[0,\infty) \r \H^3$ ending at $\a$, the set
        \[ \set{t\geq 0: \g(t) \in \G H} \]
        is unbounded.
    
    \begin{theorem} \label{K thick implies plane denseness}
    	Assume $\G$ is a geometrically finite Kleinian group containing rank-$1$ parabolic subgroups.
        Assume further that the limit set of $\G$ is a circle packing.
        Let $x \in \G\backslash \H^3$ be such that $x^- \in \L $ but $x^- \notin \partial B$ for any connected component $B\subset \partial \H^3\backslash \Lambda $.
        Suppose that $x^-$ is rank-$1$ unbounded and
        the set
            \[ R(x) = \set{t\in \R: xu_t \in \mrm{RFM}} \]
        is $K$-thick for some $K>1$. Then, the geodesic plane tangent to $x$ is dense in $M$.
        \end{theorem}

    Note that the condition on $x^-$ in the above Theorem is well-defined since $\L$ is $\G$ invariant.
    
    \begin{remark}
    We remark that, in the setting of Theorem~\ref{K thick implies plane denseness}, the limit set $\L$ is infinite.
    In this case, a result of Eberlein~\cite[Theorem 3.11]{Eberlein} states that the geodesic flow on $\mrm{T}^1(\Gamma\backslash\H^3)$, restricted to the non-wandering set, is topologically transitive.
    In particular, this implies that a dense $G_\d$ subset of the limit set is rank-1 unbounded.
    The notion of rank-1 unboundedness is also generic in the measurable sense (with respect to ergodic $\G$-invariant, fully supported conformal densities on $\L$) for similar reasons.
    \end{remark}
    
        \begin{figure}[htp]
		\centering
         \includegraphics[width=0.7\textwidth]{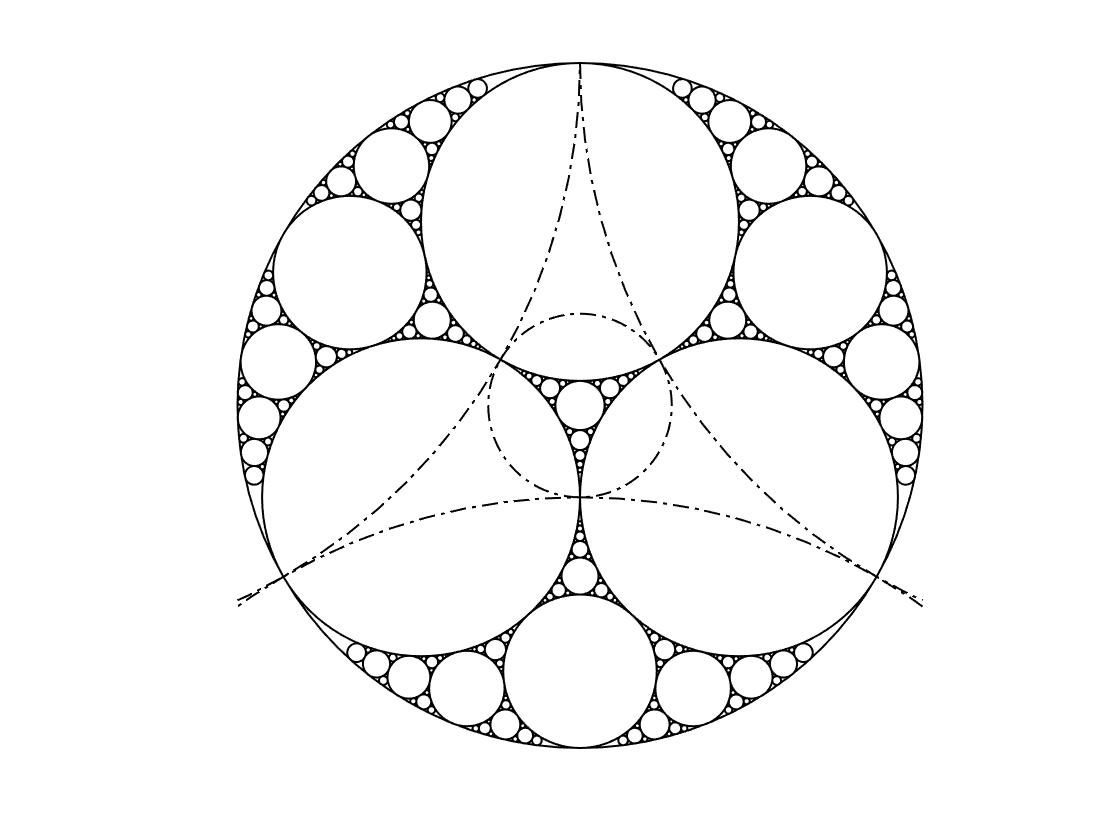}
         \caption{\label{fig:apollonian} Apollonian circle packing (solid). Inversions through dual circles (dashed) generate a geometrically finite group containing rank-$1$ parabolic subgroups.}
	\end{figure}
    
    Our second density criterion is stated in terms of orbits of circles in $\mc{C}_\L$ and does not rely on recurrence of horocycles.

	\begin{theorem} \label{accumulating on a horocycle implies dense}
    	Let $\G$ be a geometrically finite group whose limit set $\L$ is a circle packing.
		Let $C$ be a circle which meets $\L$ in at least $2$ points.
        Suppose that there exists a circle $D \in \overline{\G C}$ such that $D$ meets $\L$ in only one point. Then, $\overline{\G C} = \mc{C}_\L$.
	\end{theorem}
    We note that the $\G$ orbit of a circle is dense in $\mc{C}_\L$ if and only if the corresponding geodesic plane immersion is dense in $M$.

    \subsection{Rigidity of Closed Immersions}
    We now turn to the study of closed geodesic plane immersions.
    We show that in the setting of circle packing limit sets,
    every closed plane immersion gives rise to a geometrically finite surface.

    We further characterize the limit set of the fundamental group of such a surface.
    When $M$ is rigid acylindrical, such a limit set coincides with the intersection of the circle at $\infty$ with $\L$.
    The presence of rank $1$ parabolics, however, causes circles to meet $\L$ in non-perfect sets, while limit sets of non-elementary Fuchsian groups are known to be perfect sets. We prove that such limit sets coincide with the subset of non-isolated points of the intersection of the invariant circle with $\L$. The following is the precise statement.
    
    \begin{theorem} \label{limit set of a closed orbit}
    	Let $\G$ be a geometrically finite group whose limit set $\L$ is a circle packing.
		Let $C$ be a circle such that $\G C$ is closed in $\mc{C}$ and let $\G^C$ denote the stabilizer of $C$ in $\G$.
        Then, $\G^C$ is finitely generated and either $|\L(\G^C)| = |C\cap \L| =1 $, or $\L(\G^C)$ consists of the set of non-isolated points in $C\cap \L$.
	\end{theorem}
    
    Here, for a subgroup $\G' < \G$, the set $\L(\G')$ denotes the limit set of $\G'$. We note that Theorem~\ref{limit set of a closed orbit} implies that the geodesic plane corresponding to the circle $C$ as in the statement projects to a geometrically finite surface inside $\G\backslash\H^3$.

\subsection{Existence and Isolation of Elementary Surfaces}

  Among the features of the presence of rank-$1$ limit points in the limit set is the presence of circles on the boundary of $\H^3$ which meet the limit set in only finitely many points.
  Dual circles in Apollonian packings are examples of such circles. See figure~\ref{fig:apollonian}.
  The following Theorem says that such circles give rise to closed plane immersions.

	\begin{theorem}\label{finite intersection implies closed}
		Let $\G$ be a geometrically finite group whose limit set $\L$ is a circle packing.
		Let $C$ be a circle such that $|C\cap \L| < \infty$. Then, $\G C$ is closed.
	\end{theorem}
    
    \begin{remark}
    Recall that a surface is \emph{elementary} if its fundamental group is finite or virtually isomorphic to $\Z$.
    Theorem~\ref{finite intersection implies closed} implies that if $|C\cap \L|<\infty$, then the geodesic plane in $\H^3$, determined by $C$, projects to a closed subsurface $S$ inside $\G\backslash \H^3$ whose fundamental group is $\G^C$ (the stabilizer of $C$ in $\G$).
    Theorem~\ref{limit set of a closed orbit} thus implies that $\L(\G^C) = \emptyset$ and, in particular, $\G^C$ is finite in this case.
    Hence, the surfaces arising in the situation of Theorem~\ref{finite intersection implies closed} are elementary.
    \end{remark}

	Our motivation for studying manifolds for which the limit set is a circle packing with regards to the rigidity problem comes from the fact that groups with circle packing limit sets contain many lattice surface subgroups.
    These are subgroups which leave one of the circles of the packing invariant.
    Equivalently, such groups give rise to geometrically finite manifolds containing totally geodesic finite volume hyperbolic surfaces.
    The presence of such "rich" submanifolds was shown to be an important feature driving topological rigidity in the proofs of~\cite{MMO-Planes,Shah,MargulisOppenheim}.

	Elementary surfaces, however, pose a serious challenge to applying such key technique.
    It is thus very important to understand the distribution of such elementary surfaces within $\Gamma\backslash \H^3$.
    
    We view the cardinality of the intersection $C\cap\L$ as a measure of complexity of an elementary surface with corresponding invariant circle $C\subset \mathbb{S}^2$.
    The following proposition shows that complexity must \emph{drop} in the limit if a sequence of elementary surfaces of a given complexity accumulates on some elementary surface.
   \begin{proposition} \label{finiteness of orbits in each degree}
		Let $k \geq 3$ be an integer and let $\mc{B}_k$ denote the set of circles $C\subset \mathbb{S}^2$ such that $|C\cap \L| = k$. Then, $\mc{B}_k$ is a discrete $\G$-invariant set.
		Moreover, if a circle $C$ with $|C\cap\L|<\infty$ is an accumulation point of $\mc{B}_k$, then $|C\cap \L|<k$.
	\end{proposition}

	The paper is organized as follows: we recall some background material on geometrically finite manifolds and prove some preliminary results on the topology of the space of circles in Section~\ref{section: prelims}. The proofs of Theorems~\ref{almost every U orbit is non-recurrent},~\ref{K thick implies plane denseness},~\ref{accumulating on a horocycle implies dense} and~\ref{finite intersection implies closed} are given in Sections~\ref{section: failure of recurrence}-\ref{section: elementary orbits} respectively. In Section~\ref{section: closed orbits}, we study properly closed immersions and establish Theorem~\ref{limit set of a closed orbit}. Proposition~\ref{finiteness of orbits in each degree} is proved in Section~\ref{section: discreteness elementary}.

\section{Background and Preliminaries}
\label{section: prelims}


	\subsection{Geometrically Finite Manifolds}
    The standard reference for the material in this section is~\cite{Bowditch1993}.
	A discrete subgroup $\G < G$ of isometries of $\H^3$ is \textbf{geometrically finite} if its action on $\H^3$ admits a finite sided fundamental domain.
    A geometrically finite hyperbolic manifold is a quotient of $\H^3$ by a geometrically finite Kleinian group.
    Bowditch~\cite{Bowditch1993} proved the equivalence of this definition to the limit set of $\G$ consisting entirely of radial and bounded parabolic limit points.
    This characterization of geometric finiteness will be of importance to us and so we recall here the definitions and basic properties of all the objects involved.

    Let $o$ be any point in $\H^3$. The limit set of $\G$, denoted by $\L(\G)$, is defined to be
    \begin{align*}
    	\L(\G) = \overline{\G o}\cap \partial\H^3.
    \end{align*}
    We often use $\L$ to denote $\L(\G)$ when $\G$ is understood from context.
    $\L$ is the smallest closed $\G$ invariant set in $\partial \H^2$ and as such $\G$ acts minimally on $\L$.
    In particular, this definition is independent of the choice of $o$.
    The domain of discontinuity of $\G$, denoted by $\Omega$, is the complement of $\L$ in $\partial \H^3$.
    The action of $\G$ on $\Omega$ is properly discontinuous.
    
    A point $\xi \in \L$ is said to be a \textbf{radial point} if there exists $R>0$, a geodesic ray $l: [0,\infty] \r \H^3\cup \partial\H^3$ and a sequence $t_n \r \infty$ such that for each $n$ there exists $\g_n \in \G$ so that
    \[ d_{\H^3}(\g_n o, l(t_n)) \leq R.  \]
    In other words, any geodesic ray terminating at $\xi$ returns infinitely often to a bounded subset of $\G\backslash \H^3$.
    The set of radial limit points is denoted by $\L_r$.
    
    Denote by $N$ the following subgroup of $G$.
    \begin{align*}
    	N = \set{n_z =\begin{pmatrix} 1 & z\\ 0 & 1 \end{pmatrix}: z\in \C}.
    \end{align*}
    A point $\s \in \L$ is said to be a \textbf{parabolic point} if the stabilizer of $\s$ in $\G$, denoted by $\G_\s$, is conjugate in $G$ to a subgroup of $N$.
    This implies that $\s$ is the only fixed point under $\G_\s$ in $\partial \H^3$.
    A geodesic ray terminating at a parabolic limit point doesn't accumulate in $\G\backslash\H^3$.
    The set of parabolic limit points will be denoted by $\L_p$.
    
    Since parabolic points are fixed points of elements of $\G$, $\L$ contains only countably many such points.
    Moreover, $\G$ contains at most finitely many conjugacy classes of parabolic subgroups.
    This translates to the fact that $\L_p$ consists of finitely many $\G$ orbits.
    
    The \textbf{rank} of a parabolic point is the rank of its stabilizer $\G_\s$ which is an abelian group.
    A rank $1$ parabolic fixed point in $\L$ will be referred to as a rank $1$ \textbf{cusp}.
    
    A parabolic limit point $\s$ is said to be \textbf{bounded} if $\G \backslash \left(\L-\set{\s} \right)$ is compact.
    All full rank (rank $2$) parabolic points are bounded.

    The key geometric property of a bounded rank-$1$ parabolic limit point of $\G$ is that it admits a double horocycle.
    The following is the precise definition.
	\begin{definition}
		Let $\s\in \L$ be a rank-$1$ parabolic fixed point. A pair of circles $L_1,L_2\subset \partial \H^3$ is said to be a \textbf{double horocycle} at $\s$ if $L_1 \cap L_2 = \set{\s}$ and if they bound open disks which are completely contained in $\Omega$.
	\end{definition}


\subsection{Cusp Neighborhoods}  \label{cuspnbhd}  
     Let $\set{\s_1, \dots, \s_s}\subset \partial \H^3$ be a maximal set of nonequivalent parabolic fixed points under the action of $\G$. 
     As a consequence of geometric finiteness of $\G$, one can find a finite disjoint collection of \textbf{open} horoballs $H_1, \dots, H_s \subset \H^3$ with the following properties (cf.~\cite{Bowditch1993}):
    
    \begin{enumerate}[(a)]
    	\item $H_i$ is centered on $\s_i$, for $i = 1,\dots, s$.
	    \item $\G \overline{H_i} \cap \G \overline{H_j} = \emptyset$ for all $i\neq j$.
        \item For each $i \in \set{1,\dots,s}$, $\g_1 \overline{H_i} \cap \g_2 \overline{H_i} = \emptyset$ for all $\g_1,\g_2 \in \G$, $\g_1 \neq \g_2$.
    \end{enumerate}
    
    Let $\mathcal{H} = \pi^{-1}(\bigcup_{i=1}^s \G H_i) \subset \mrm{F}\H^3$, where $\pi: \mrm{F}\H^3\r \H^3$ denotes the standard projection.
    We shall refer to $\mc{H}$ as a \textbf{cusp neighborhood}.

    
    \subsection{Topology of the space of circles}
   	As discussed in the introduction, our arguments will take place in the space of circles $\mc{C}= \mrm{PSL}_2(\C)/\mrm{PGL}_2(\R)$ on the boundary.
    In this section, we prove some simple facts about the topology of this space which will be useful for us.
    We shall need the following
    
	\begin{definition} \label{annulus}
		Let $C$ be a circle, $\eta$ be the center of one of the disks bounded by $C$ and $r$ be the radius of that disk. For $\e\in(0,r)$, let $N_\e(C)$ be defined as follows:
        \[ N_\e(C) = \set{x\in \mathbb{S}^2: d_{\mathbb{S}^2}(x,\eta) \in (r-\e,r+\e) }. \] 
        Then, the $\e$-\textbf{annulus} around $C$ is defined to be
        \[ \mc{N}_\e(C) = \set{ D: D \text{ a circle in } \mathbb{S}^2, D\subset N_\e(C) }. \]
	\end{definition}
	
    Note that the definition of $N_\e(C)$ is independent of the choice of the disk bounded by $C$ and hence the choice of $\eta$. Thus, $\mc{N}_\e(C)$ is well-defined. This definition provides a convenient description of a neighborhood of a circle as explained in the following simple proposition.
    
    \begin{proposition} \label{annuli are open}
		Every $\e$-annulus around a circle $C$ contains an open neighborhood in the space of circles.
	\end{proposition}
    
    \begin{proof}
		Recall from the proof of Proposition~\ref{circles meeting cpt set} that there is a continuous $2$-$1$ covering map
        \[\vp: \mathbb{S}^2 \times (0,\pi) \r \mc{C},\]
        where $\mc{C}$ is the space of circles in $\mathbb{S}^2$. Let $\mc{N}_\e(C)$ be an $\e$-annulus around $C$ for some $\e$. Let $\d \in (0,\e)$. Let $\eta$ be the center of one of the disks bounded by $C$ and let $r>0$ be its radius. We shall show that
        \[ \vp [B(\eta,\d) \times (r-\e + \d,r+\e-\d) ] \subseteq \mc{N}_\e(C), \]
        where $B(\eta,\d)$ is the open disk of radius $\d$ around $\eta$ in $\mathbb{S}^2$ in the spherical metric. Note that we may assume that $\e$ is small enough so that $(r-\e + \d,r+\e-\d) \subseteq (0,\pi)$.
        
        To see this, let $x\in B(\eta,\d)$ and $d\in (r-\e + \d,r+\e-\d)$. Let $D = \vp(x,d)$ be the circle bounding the disk of radius $d$ around $x$. Let $y\in D$. Then, $d_{\mathbb{S}^2}(y,x) =d$. The following $2$ inequalities show that $y \in N_\e(C)$ (Def.\ref{annulus})
        \[ d_{\mathbb{S}^2}(y,\eta) \leq  d_{\mathbb{S}^2}(y,x) + d_{\mathbb{S}^2}(x,\eta) < r+\e-\d + \d =r+\e,  \]
        \[ d_{\mathbb{S}^2}(y,\eta) \geq d_{\mathbb{S}^2}(y,x) - d_{\mathbb{S}^2}(x,\eta) > d-\d > r-\e.\]
	\end{proof}
    
    \begin{proposition} \label{circles meeting cpt set}
		Let $K \subset \H^3$ be a compact set. Then, the set
        \[  \mc{C}(K) = \set{C: \mathrm{hull}(C) \cap K \neq \emptyset} \]
        is compact.
	\end{proposition}
    
    \begin{proof}
		Let $\mc{C}$ denote the space of circles in $\mathbb{S}^2$. It is clear that $\mc{C}(K)$ is closed.
        Thus, it suffices to prove this in the case where $K$ is a closed ball since any compact set is contained in a closed ball. Moreover, it suffices to take $K$ to be a closed ball centered around the origin in the ball model of $\H^3$. That is because for any $g\in G = \mrm{Isom}(\H^3)$, it is easy to see that $\mc{C}(gK) = g\mc{C}(K)$ and $G$ acts continuously on $\mc{C} \cong G/\mrm{PGL}_2(\R)$.
        
        Next, note that there is a continuous $2$-to-$1$ covering map from $\vp: \mathbb{S}^2 \times (0,\pi) \r \mc{C}$ given by sending $(x,r)$ to the boundary circle of the Euclidean disk in $\mathbb{S}^2$ centered at $x$ and of radius $r$. Hence, it suffices to prove that the pre-image of $\vp^{-1}(\mc{C}(K))$ is compact.
        
        Since $K$ is a compact ball centered around the origin in the ball model, for every $x\in \mathbb{S}^2$, there exists a compact subinterval $I\subset (0,\pi)$ such that
        \[\set{x} \times I = \set{x} \times (0,\pi) \cap \vp^{-1}(\mc{C}(K)).\]
        Moreover, from symmetry, the interval $I$ is independent of $x$. Hence, $\vp^{-1}(\mc{C}(K)) = \mathbb{S}^2 \times I$ which is compact as desired.
	\end{proof}


\section{Failure of Recurrence of Unipotent Orbits}
\label{section: failure of recurrence}

    This section is dedicated to the proof of Theorem~\ref{almost every U orbit is non-recurrent}.
    We shall need some technical lemmas before the proof.

\subsection{Some hyperbolic geometry}
	For background on hyperbolic geometry, we refer the reader to~\cite[Chapter A]{BenedettiPetronio}.
	Throughout this section, we assume $\G$ to be a torsion-free geometrically finite Kleinian group containing rank-$1$ parabolic subgroups and we use $\L$ to denote its limit set.
    We will need the following notion of a shrinking neighborhood of a cusp.
    \begin{definition}
    	A sequence of closed horoballs $H_n$ centered at a point $\s\in \partial \H^3$ is said to be shrinking to $\s$ if $H_{n+1}\subset H_n$ and given any point $o\in \H^3$, $d_{\H^3}(o, H_n) \r \infty$.
    \end{definition}

        For any two points $a,b\in \partial \H^3$, let $l(a,b)$ denote the geodesic joining $a$ and $b$.
        Given two geodesic lines $l_1, l_2 \subset \H^3$, the hyperbolic distance between $l_1$ and $l_2$ is defined as follows.
        \begin{equation*}
        d_{\H^3}(l_1,l_2) = \inf\set{d_{\H^3}(x_1,x_2): x_1\in l_1, x_2\in l_2  }.
        \end{equation*}
        When $l_1$ does not meet $l_2$ on the boundary of $\H^3$ is that the infimum in the above definition is realized by a unique pair of points $x_i \in l_i$, $i=1,2$.
        This is a consequence of the strict convexity of the distance function on $\H^3$.

        \begin{figure}
          \centering
          \includegraphics[width=0.5\textwidth]{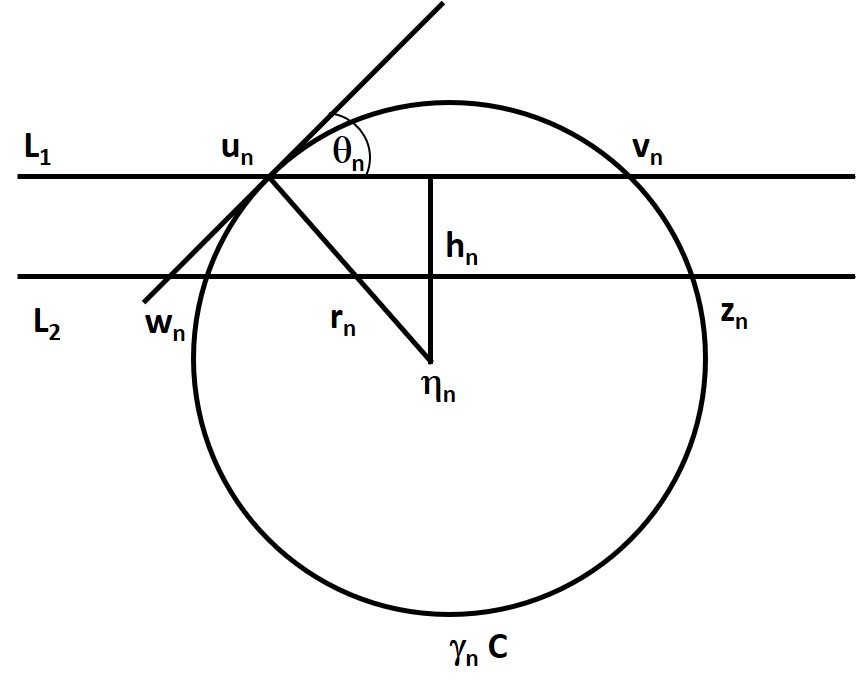}
          \caption{\label{fig:symunbndd} Proof of Lemma~\ref{angles lemma} }
        \end{figure}
	
    \begin{lemma} \label{angles lemma}
    	Let $\s \in \L(\G)$ be a rank-$1$ parabolic point and let $L_1$ and $L_2$ be a double horocycle at $\s$.
        Let $C\subset \partial \H^3$ be a circle.
        Suppose that there exists a sequence $\g_n \in \G$ such that $|\g_n C \cap L_i| = 2$, for $i=1,2$.
        For each $n$, let $\set{u_n,v_n} = \g_n C \cap L_1$ and let $\set{w_n,z_n} = \g_n C \cap L_2$.
        Assume further that the following holds.
       \begin{enumerate}
       	\item $\s\notin \g_n C$ for all $n\geq 1$.
        \item There exists some $\d>0$ such that for all $n$, if we let $\theta_n \in[0,\pi/2]$ denote the angle between $\g_n C$ and $ L_1$, then $\theta_n > \d$.
        \item There exists a sequence of horoballs $H_n$ shrinking to $\s$ such that $l(u_n,v_n) \cap H_n \neq \emptyset$ for all $n$.
       \end{enumerate}
        Then,
        \[ d_{\H^3}(l(u_n,v_n), l(w_n,z_n)) \r 0. \]
    \end{lemma}
    
    Before proceeding to the proof, let us motivate the statement of the lemma.
        In~\cite[Lemma 9.2]{MMO-Planes}, the proof of the $K$-thickness property was derived from a certain isolation property of the boundary circles in the limit set.
        Lemma~\ref{angles lemma} provides us with concrete sufficient conditions for the failure of such isolation property in the presence of rank-$1$ parabolics.
        Using the assumption on the density of the geodesic $x$ in Theorem~\ref{almost every U orbit is non-recurrent}, we will be able to arrange for the $\G$ orbit of the circle corresponding the plane defined by $x$ to satisfy the geometric configuration given by the sufficient conditions in Lemma~\ref{angles lemma}.
    
    \begin{proof}[proof of Lemma~\ref{angles lemma}]
    Choose coordinates in the upper half space such that $\s = \infty$. In these coordinates, $L_1$ and $L_2$ become two parallel lines in $\R^2$.
    Without loss of generality, we will assume that the Euclidean distance between $L_1$ and $L_2$ is $1$.
    Let $n\in \N$. By assumption, $\s\notin \g_n C$. Hence, in these coordinates, $\g_n C$ is a Euclidean circle in $\R^2$.
    
    Let $\eta_n$ be the Euclidean center of the circle $\g_n C$ and let $r_n$ be its Euclidean radius.
    Let $h_n\geq 0$ be the vertical distance from $\eta_n$ to the line $L_1$.
    Let
        \[d_n = |u_n - v_n|,\]
    where $|\cdot|$ denotes the Euclidean norm.
    See figure~\ref{fig:symunbndd}.
        
    Note that since $\theta_n > \d$, we get that the ratio $\frac{h_n}{r_n} = \sin \theta_n$ is bounded away from $0$ for all $n$.
	Moreover, since $l(u_n,v_n)$ meets $H_n$, we have that 
        \[d_n \xrightarrow{n\r\infty} \infty. \]
    But, then, since $r_n \geq d_n/2$, one has that as $n\r \infty$, $r_n \r\infty$. Hence, we get $h_n \r \infty$ as $n\r\infty$.
    
    For all $n\geq 1$, let $\a_n$ (resp. $\b_n$) denote the highest point on the geodesic $l(u_n,v_n)$ (resp. $l(w_n,z_n)$).
    Note that
       	\begin{align}
       		\label{eqn: distance of geodesics bounded by distance of pts}
            d_{\H^3}(l(u_n,v_n) , l(w_n,z_n))
            	\leq d_{\H^3}(\a_n, \b_n).
       	\end{align} 
	Using the triangle inequality, the distance $d_{\H^3}(\a_n, \b_n)$ can be estimated using the sum of vertical (parallel to the $z$-axis) and horizontal (parallel to $\R^2$) distances in the upper half space.
    Recall that the hyperbolic distance between two points lying on the same vertical line in upper half space is given by the absolute value of the logarithm of the ratio of their heights.    
    Since $l(u_n,v_n)$ is a half circle which is perpendicular which is perpendicular to $\R^2$, we see that the height of $\a_n$ is $d_n/2$. Similarly, the height of $\b_n$ is $|z_n - w_n|/2$.
    
    Moreover, the hyperbolic distance between two points $(x_1,y_1,z)$ and $(x_2,y_2,z)$ at the same height $z>0$ in the upper half space is given by 
    \begin{align}
    \label{eqn: vertical distance bound}
    d_{\H^3}( (x_1,y_1,z),(x_2,y_2,z))  = \frac{1}{2} \ln \left( \frac{d^2}{z^2} +1\right),
    \end{align} 
    where $d$ is the Euclidean distance between $(x_1,y_1)$ and $(x_2,y_2)$.
    This can be seen by considering the vertical plane containing the two points and doing the calculation in the upper half plane instead.

    Now, note that in our chosen coordinates, the horospheres bounding our sequence of shrinking horoballs are horizontal planes parallel to $\R^2$.
    Let $k_n$ denote the height of the boundary of $H_n$. Then, since $H_n$ are shrinking to $\s$, $k_n \r\infty$.
    Moreover, the heights of $\a_n$ and $\b_n$ are at least $k_n$. Combining the estimates of the horizontal and vertical distances, we get
         \[ d_{\H^3}(\a_n, \b_n) \leq \frac{1}{2}\ln\left(1+ \frac{1}{k_n^2} \right)   
         	+ \left| \ln\left(\frac{|z_n - w_n|}{d_n}\right) \right|. \]
        Elementary geometric considerations show that since $\theta_n$ is bounded away from $0$, the ratio $\frac{|z_n - w_n|}{d_n}$ tends to $1$ as $n\r \infty$.
        Thus, by~\eqref{eqn: distance of geodesics bounded by distance of pts}, we get the desired conclusion.
    \end{proof}

    Our next Lemma shows how one can use Lemma~\ref{angles lemma} to prove the lack of $K$-thickness.
    \begin{lemma} \label{symmetrization lemma}
    	Let $x\in \mrm{F}\H^3$ be such that $x^-\in \L(\G)$.
        Let $C \subset \partial \H^3$ be the circle determined by the geodesic plane which is tangent to the first two components of $x$.
        Suppose that there exist 2 distinct components $B_1$ and $B_2$ of the domain of discontinuity $\Omega$ and a sequence $\g_n \in \G$ such that the following holds: for each $n$, there exist points $u_n,v_n \in  C \cap \g_n\overline{B_1}$ and
        $w_n,z_n \in  C \cap \g_n\overline{B_2}$ satisfying the following
        \begin{enumerate}
        \item $u_n$ and $v_n$ bound an open interval in $C$ which is fully contained in $B_1$.
        \item $w_n$ and $z_n$ bound an open interval in $C$ which is fully contained in $B_2$.
        \item $u_n, v_n$ lie in a different connected component of $C-\set{x^\pm}$ than that of $w_n,z_n$,
        \item $d_{\H^3}(l(u_n,v_n), l(w_n,z_n)) \r 0$,
        \item $u_n,v_n,w_n,z_n \xrightarrow{n\r\infty} x^-$.
        \end{enumerate}
        Then, the set
        \[ R(x) = \set{u\in U: xu\in \mrm{RFM}} \]
        is not $K$-thick for any $K>1$.
    \end{lemma}
    
    \begin{proof}
        	
    	Choose coordinates in the upper half space such that $x^- = \infty$, $x^+=0$,
        and $C=\widehat{\R}$.
        We may assume that $\pi(x) = (0,0,1)$, where $\pi: \mrm{F}\H^3 \r \H^3$ denotes the natural projection.
        In these coordinates, we have that $(xu(t))^+ = t$.
        Hence, the set $R(x)$ may be identified with the points where $C$ meets $\L$.
        Let $K > 1$. We will show that there exists a sequence $t_j \r \infty$ such that for all $j$:
        \begin{align}
        	\label{eqn: R(x) not K thick}
        	R(x) \cap \left( [-Kt_j,-t_j] \cup [t_j,Kt_j] \right) = \emptyset.
        \end{align}  
        
        Without loss of generality, we may assume that $v_n>u_n > 0$ and $z_n<w_n < 0$ for all $n$.
        We may also assume that
            \[ |w_n| \geq u_n. \]
        By assumptions $(1)$ and $(2)$ in the statement, we have that
            \[ (z_n, w_n)\cup (u_n,v_n) \subset \Omega, \]
        where $\Omega$ is the domain of discontinuity.
        Now, for each $n$, we have
            \begin{align} \label{base estimate}
            	d_{\H^3}(l(u_n,v_n), l(w_n,\infty)) \leq d_{\H^3}(l(u_n,v_n), l(w_n,z_n)).
            \end{align}
            
        \textbf{Case 1:} For all $n$, $|z_n| \geq v_n$.
        By applying a translation by $-w_n$, which is a hyperbolic isometry, the left hand side of inequality~\eqref{base estimate} becomes
        \[ d_{\H^3}(l(u_n,v_n), l(w_n,\infty)) 
        			= d_{\H^3}( l(u_n-w_n,v_n-w_n), l(0,\infty)). \]
            
        Applying the isometry $z\mapsto \frac{z}{u_n-w_n}$, we get
          \[ d_{\H^3}( l(u_n-w_n,v_n-w_n), l(0,\infty)) 
        	= d_{\H^3}\left( l\left(1,\frac{v_n-w_n}{u_n-w_n}\right), l(0,\infty)\right). \]
                        
        Since the right hand side of inequality~\eqref{base estimate} tends to $0$ by condition $(4)$ in the statement of the lemma, we necessarily have that
        \[ \frac{v_n-w_n}{u_n-w_n} = \frac{1+ \frac{v_n}{|w_n|}}{1+ \frac{u_n}{|w_n|}} \r \infty. \]
            
        Therefore, since $|w_n| \geq u_n$ by assumption, we get that
            \begin{align} \label{ratio of endpoints}
            	\frac{v_n}{|w_n|} \r \infty.
            \end{align}
            
        Moreover, we have that
            \begin{align} \label{symmetric intervals}
            	(-v_n, w_n)\cup (|w_n|,v_n)\subseteq (z_n, w_n)\cup (u_n,v_n) \subset \Omega.
            \end{align}
       Thus, for all $n$ sufficiently large, we have that $K < v_n/|w_n| $ and hence we get 
       \[  (Kw_n, w_n)\cup (|w_n|,K|w_n|) \subset (-v_n, w_n)\cup (|w_n|,v_n) \subseteq \Omega. \]
       This proves~\eqref{eqn: R(x) not K thick} by taking $t_n =  |w_n|$ for $n$ sufficiently large.
       Condition $(5)$ in the statement of the lemma guarantees that $|w_n| \r \infty$ as $n\r\infty$.
            
        \textbf{Case 2:} For all $n$, $|z_n| < v_n$.
        Note that we have
            \[ d_{\H^3}(l(0,\infty), l(w_n,z_n)) \leq d_{\H^3}(l(u_n,v_n), l(w_n,z_n)) \r 0. \]
            
        Applying the isometry $z\mapsto \frac{z}{|w_n|}$, we get
            \[ d_{\H^3}(l(0,\infty), l(w_n,z_n)) = d_{\H^3}\left(l(0,\infty), l\left(-1,\frac{z_n}{|w_n|}\right) \right) \r 0.  \]
            
        Hence, we get that
            \[ \frac{z_n}{w_n} \r \infty. \]
            
        But, by assumption, we have that
            \[ (z_n, w_n)\cup (|w_n|,|z_n|)\subseteq (z_n, w_n)\cup (u_n,v_n) \subset \Omega. \]
            
        Thus, arguing as in Case 1, we get that $R(x)$ is not $K$-thick for any $K$.
    \end{proof}

\subsection{Proof of Theorem~\ref{almost every U orbit is non-recurrent}}

    	Let $x\in \mrm{RFM}$ be a dense frame for the $A$ action on $\mrm{RFM}$.
        Our strategy will be to use the denseness of $xA$ to find elements of $\G$ to put the circle defined by $x$ into configurations satisfying the assumptions of Lemma~\ref{angles lemma}. This will allow us to verify condition $(1)$ in Lemma~\ref{symmetrization lemma}.
        Along the course of the proof, we will insure that these configurations also satisfy the other conditions in Lemma~\ref{symmetrization lemma} which will allow us to conclude the lack of $K$-thickness.
        
        By abuse of notation, we also use $x\in \mrm{F}\H^3$ to denote a lift of $x$.
        Let $\s\in \L(\G)$ be a rank-$1$ parabolic fixed point.
        Let $L_1$ and $L_2$ be a double horocycle at $\s$.
        Choose coordinates in the upper half space model of $\H^3$ in which $\s = \infty$.
        Without loss of generality, we may assume $L_1 = \set{z\in \C: \mathrm{Im}(z) = -1}$ 
        and $ L_2 = \set{z\in \C: \mathrm{Im}(z) = 1}$.        
        Let $S$ be the strip defined by
        \[ S = \set{z\in \C: -1< \mathrm{Im}(z) < 1}. \]
        
        Let $\G_\s\subset \G$ denote the stabilizer of $\s$.
        Then, $\G_\s$ acts by horizontal translations.
        Without loss of generality, we may assume $\G_\s$ is generated by $z \mapsto z+1$.
        Let $F$ be a fundamental domain for the action of $\G_\s$.
        Let $\l_0 \in F\cap \L$ and let $\l_n = \l_0 + n$ for $n\geq 1$.
        Since $\L$ is $\G$-invariant, we have that $\l_n \in \L$ for all $n$.
        
        Consider the sequence of circles $C_n$ centered at $\l_0+n/2$ and passing through $\l_0$ and $\l_n$, for $n\geq 1$.
        Let $y_n \in \mrm{F}\H^3$ be a frame whose first two vectors are tangent to to the geodesic plane defined by $C_n$ and satisfying the following conditions.
        \begin{equation*}
        y_n^+=\l_n,\quad  y_n^-=\l_0, \quad \pi(y_n) = (0,\l_0+n/2, n/2).
        \end{equation*}
         Here $\pi:\mrm{F}\H^3\r \H^3$ is the natural projection.
        
        Fix $\e>0$ such that $\e \ll 1$. For each $n$, let $N_\e(C_n)$ denote the open $\e$-annulus around $C_n$ (Definition~\ref{annulus}).
        Define the following neighborhood of $C_n$:
        \[ \mc{N}_\e(C_n) = \set{C\subset\C:C\subset N_\e(C_n) }. \]
        
        For each $n\geq 0$, let $B_\e(\l_n)$ denote the Euclidean ball of radius $\e$ around $\l_n$
        and let $B_\e^{\H^3}(\pi(y_n))$ denote the ball of radius $\e$ around $\pi(y_n)$ in the hyperbolic metric on $\H^3$.
        For each $y\in \mrm{F}\H^3$, let $C(y)$ denote the circle defined by the plane to which $y$ is tangent.   
        
        Let $E_n \subset \mrm{F}\H^3$ be an open neighborhood of $y_n$ in $\mrm{RFM}$ defined as follows:
        \[ E_n = \left\lbrace        	
        	y\in \mrm{F}\H^3: \begin{array}{ll}
            y^-\in B_\e(\l_n), y^+\in B_\e(\l_0), \\
            C(y) \in \mc{N}_\e(C_n), \pi(y)\in B_\e^{\H^3}(\pi(y_n))\end{array}
        	\right\rbrace.
        \]
        Since each $E_n$ is open in $\mrm{RF}\H^3$ and since the $A$-orbit of $x$ is dense in $\mrm{RFM}$ by assumption, then, for each $n$, we can find $\g_n \in \G$ and $a_n\in A$ such that
        \[ \g_n x a_n \in E_n. \]
        
        \subsubsection*{\textbf{Claim 1}} The circle $C(x)$ and the sequence $\g_n$ satisfy the hypotheses of
        Lemma~\ref{angles lemma}.
        
        First note that $C_n \cap L_i \neq \emptyset$ for $i=1,2$ for all $n\geq 2$.
        Hence, since $\g_n C(x) \in \mc{N}_\e(C_n)$ and $\e<<1$, we see that
        $|\g_n C(x) \cap L_i| = 2$ for $i=1,2$ and all $n\geq 2$.
        Moreover, clearly $\s = \infty \notin \g_n C(x)$ for all $n$.

        Next, let $\rho_n \in [0,\pi/2]$ denote the angle between $C_n$ and $L_1$. Let $a$ denote the Euclidean distance from $\l_0$ to $L_1$. Then, since the center of $C_n$ is $\l_0 + n/2$, we see that $\cos(\rho_n) = 2a/n$.
        Hence, $\rho_n \r \pi/2 $ and in particular is bounded away from $0$ as $n\r \infty$ and therefore  
        there exists some $\delta > 0$, depending only on $\e$ and $a$, such that for all $n\geq 1$ and for all $C\in \mc{N}_\e(C_n)$, the angle in $[0,\pi/2]$ between $C$ and $L_1$ is at least $\d$.
        This verifies condition $(2)$ of the Lemma.
        
        To verify the last condition, let $\set{u_n',v_n'} = C_n \cap L_1$.
        Then, $\sin (\rho_n) = 2a/|u_n'-v_n'|$, where $|\cdot|$ denotes the Euclidean distance.
        Since $\rho_n \r \pi/2$, we see that $|u_n'-v_n'| \r \infty$.        
        Let $\set{u_n,v_n} = \g_n C(x) \cap L_1$ and $\set{w_n,z_n} = \g_n C(x) \cap L_2$.
        Then, since $\g_n C(x) \in \mc{N}_\e(C_n)$, we see that 
        \begin{align} \label{eqn: u_n - v_n tends to infinity}
        	|u_n -v_n| \r\infty.
        \end{align}
        
        Let $H_n$ be the horoball centered at $\s = \infty$ and bounded by the horizontal plane parallel to $\R^2$ at height $|u_n-v_n|/2$.
        Then, $H_n$ shrinks to $\s$ and $l(u_n,v_n) \cap H_n \neq \emptyset $ for all $n$.
        Thus, condition $(3)$ is verified and we conclude by Lemma~\ref{angles lemma} that
        \begin{align*}
        	d_{\H^3}(l(u_n,v_n), l(w_n,z_n)) \r 0.
        \end{align*}
        
        \subsubsection*{\textbf{Claim 2}} $x$ satisfies the hypotheses of Lemma~\ref{symmetrization lemma}.
        
        Note that $\L \subset S$. Hence, $\set{|\mathrm{Im}(z)| > 1 } \subset \Omega$.
        This implies conditions $(1)$ and $(2)$.
        Moreover, since $\g_n x^+, \g_n x^- \in S$, condition $(3)$ is also verified.
        Condition $(4)$ follows by Claim $1$.
        
        Let $\tau_n \in \G_\s$ be the hyperbolic isometry acting by horizontal translation satisfying that the Euclidean center of the circle $\tau_n \g_nC(x)$ belongs to the fundamental domain $F$ for the action of $\G_\s$.

        Then, by~\eqref{eqn: u_n - v_n tends to infinity}, we see that $\tau_n u_n $ and $\tau_n v_n$ tend to $\s$ as $n\r\infty$.
        Since the width of the strip $S$ is bounded, it is not hard to see that $\tau_n w_n $ and $\tau_n z_n$ also tend to $\s = \infty$ as $n\r\infty$.
        Thus, condition $(5)$ is verified and we conclude that $R(x)$ is not $K$-thick for any $K$ as desired.

\section{Recurrence of Horocycles and Rigidity of Planes}

        In this section, we give a proof of Theorem~\ref{K thick implies plane denseness}.
        We will use similar techniques to those developed in the previous section.
        Throughout this section, we assume that $\G$ is a torsion-free geometrically finite Kleinian group containing rank-$1$ parabolic subgroups.
        We further assume that the limit set of $\G$ is a circle packing i.e. that the complement of $\L$ in $\mathbb{S}^2$ consists of countably many round open disks.
        
        We recall the notion of rank-$1$ unboundedness defined in the introduction.
        
        \begin{definition} \label{definition unbounded point}
    	A radial limit point $\a \in \L$ is said to be \textit{rank}-$1$ \textit{unbounded} if there exists a rank-$1$ parabolic fixed point $\s$ such that for every horoball $H$ centered at $\s$ and every geodesic ray $\g:[0,\infty) \r \H^3$ ending at $\a$, the set
        \[ \set{t\geq 0: \g(t) \in \G H} \]
        is unbounded.
    \end{definition}

    The following result obtained in~\cite{MMO-Planes} will be useful for us. We refer the reader to~\cite{MargulisOppenheim} and~\cite{Shah} for versions of this result in the finite volume setting.
    Recall the $C_\L$ denotes the set of circles $C\subset\partial\H^3$ such that $C\cap \L \neq \emptyset$.
    
    \begin{theorem} [Theorem 5.1 in~\cite{MMO-Planes}] \label{accumulating on the boundary of a component}
    Assume $\g_n C \r E$ where $E$ is a circle whose orbit is closed and $\G^E$ is non-elementary. Assume further that for infinitely many $n$, $\g_n C$ meets a disk bounded by $E$ in $\mathbb{S}^2$ which meets $\L$. Then, $\overline{\G C} = \mc{C}_\L$.

	\end{theorem}
    
     It will be convenient for us to introduce the notion of the height of a point in $\H^3$ with respect to a point on the boundary.
     Suppose $\s \in \partial\H^3$.
     Given a choice of identification $\H^3\cong \C^2\times \R_{>0}$ as the upper half space model sending $\s$ to $\infty$, define a map $\mathrm{height}_\s:\H^3\r \R_{>0}$ to be the projection on the $\R_{>0}$ factor. More precisely,
     \begin{equation}\label{defn: height}
     	\mathrm{height}_\s(z,t) = t.
     \end{equation}
     
     \begin{remark} The following properties of the height function defined above will be used in the proof.
     \begin{enumerate}
     \item The notion of height introduced above depends on the choice of coordinates $\H^3 \cong \C^2 \times \R^+$.
     However, for any two points $x,y \in \H^3$, the inequality $\mathrm{height}_\s(x)\geq \mathrm{height}_\s(y)$ remains valid in any choice of coordinates sending $\s$ to $\infty$.
     
     \item In a choice of coordinates on $\H^3$ in the upper half space model sending $\s$ to $\infty$, parabolic isometries of $\H^3$, which fix $\s$, act by translations parallel to the complex plane.
     In particular, these parabolic isometries preserve the level sets of $\mathrm{height}_\s$ which are horospheres based at $\s$.
     \end{enumerate}
     
     \end{remark}

        \begin{proof}[Proof of Theorem~\ref{K thick implies plane denseness}]
        	Let $x$ be as in the statement and let $C$ be a boundary circle corresponding to the geodesic plane defined by the frame $x$.
            By definition of rank-$1$ unboundedness, there exists a rank-$1$ parabolic fixed point $\s \in \L$ such that for all horoballs $H$ centered at $\s$, the geodesic ray $r:[0,\infty)\r \H^3$ joining $\pi(x)$ to $x^-$ has the property that the set
            \[ \set{t\geq 0: r(t)\in \G H} \]
            is unbounded.
            Since $\G$ is geometrically finite whose limit set is a circle packing, we can find components $B_1,B_2 \subset \Omega$ such that
            \[ \partial B_1 \cap \partial B_2 = \set{\s}. \]
            Let $L_i = \partial B_i - \set{\s}$, for $i=1,2$.
            Choose coordinates in the upper half space in which $\s = \infty$ and 
            \[ L_1 = \set{z\in \C: \mathrm{Im}(z) = 0}, \qquad  L_2 = \set{z\in \C: \mathrm{Im}(z) = 1}. \]
            Let $S$ denote the strip defined as follows:
            \[ S = \set{z\in \C: \mathrm{Im}(z)\in (0,1)}. \]
            
            Thus, $\L- \set{\s}$ is contained in $S$. Let $n>0$.            
            Let $H_n$ be the horoball centered at $\s = \infty$ whose boundary horosphere is at height $n$.
            By the unboundedness of $x^-$, for all $n>0$, there exists $\g_n \in \G$ and $t_n>0$ such that
            $t_n \r \infty$ and 
            \begin{align} \label{eqn: geodesic meets H_n}
            	\g_n r(t_n) \in H_n. 
            \end{align}
                       
            Hence, the sequence $\set{\g_n C:n\geq 1}$ consists of Euclidean circles in $\R^2$ whose Euclidean radii tend to $\infty$ or straight lines.
            And, since $\g_n x^- \in \g_n C \cap \L - \set{ \s}$, we see that $\g_n C$ meets $S$ for all $n\geq 1$.
            Thus, $\g_n C$ must meet at least one of $L_1$ or $L_2$ for all $n \gg 1$. 
            Without loss of generality, suppose that $\g_n C \cap L_1 \neq \emptyset$ for all $n\gg 1$.
            For each $n$, let $\theta_n\in[0,\pi/2]$ denote the angle between $\g_n C$ and $L_1$.
            \subsubsection*{\textbf{Case 1}} $\g_n C \cap L_2 = \emptyset$ for all $n\gg 1$ or $\th_n \r 0$.
            
            Note that $\g_n x^- \in \g_n C \cap S$ and, in particular, $\g_n C \neq L_1$.
            We will show that there exists a sequence $\g_n'$ in $\G$ such that $\g_n' C \neq L_1$ and
            $\partial B_1 = L_1\cup\set{\s} \in \overline{\set{\g_n' C}}$.
            Hence, by Theorem~\ref{accumulating on the boundary of a component} (Theorem 5.1 in~\cite{MMO-Planes}), we conclude that $\G C$ is dense in $\mc{C}_\L$.

            Let $\G_\s$ be the stabilizer of $\s$ in $\G$.
            Then, in our chosen coordinates, $\G_\s$ acts on $\R^2$ by horizontal translations parallel to the real axis.
            Let $F$ denote a fixed fundamental domain for the action of $\G_\s$. So, $F$ is an infinite vertical strip with finite width.
            
            Let $u_n \in \g_n C \cap L_1$.
            For each $n$, let $\t_n \in \G_\s$ be such that $\t_n u_n \in F$.
            The circles $\t_n \g_n C$ meet the fixed compact set $ F \cap L_1 $ and hence must accumulate.
            As we noted before, $\t_n\g_n C$ are either straight lines or circles whose Euclidean radii tend to $\infty$.
            Hence, any accumulation point of $\set{\t_n \g_n C}$ must be a straight line which meets $L_1$ at a point in $L_1 \cap F$.
            Thus, if $\th_n \r 0$, we can see that $\t_n \g_n C \r L_1$. 
            
            On the other hand, if $\t_n \g_n C \cap L_2 = \emptyset$ for $n\gg 1$, any limiting straight line $L$ cannot meet $L_2$, while $L \cap L_1 \cap F \neq \emptyset$. Thus, we get that $\t_n \g_n C \r L_1$ in this case as well, which concludes the proof in this case.

            \subsubsection*{\textbf{Case 2}} After possibly passing to subsequence, $\g_n C \cap L_2 \neq \emptyset$ and $\th_n$ is bounded away from $0$ for all $n\gg 1$.                       
            We show that this case is incompatible with our assumption on $K$-thickness.            
            As a first step, we show that the sequence $\g_n C$ satisfies the hypotheses of Lemma~\ref{angles lemma}.
            
            To that end, we claim that $\g_n x^+ \neq \s$ for $n\gg 1$. Suppose otherwise.
            Then, we get (after passing to a subsequence if necessary) that $\g_{m} \g_{n}^{-1} \in \G_\s $ for $m,n \gg 1$.
            In particular, we see that $\g_n C$ and $\g_m C$ differ by a horizontal translation.
            Recall that the unboundedness of $x^-$ allowed us to find a sequence $t_n \r \infty$ such that $r(t_n) \r x^-$ while $\g_n r(t_n) \r \s=\infty$.
            Fix $n \gg 1$ so that for all $m >n$, $\g_{m} \g_{n}^{-1} \in \G_\s $.
            As $m$ tends to infinity, we see that $\g_n r(t_m) \r \g_n x^-$.
            
            Moreover, by assumption, $x^-$ is a radial limit point while $\s$ is a parabolic point.
            This implies that $\g_n x^- \neq \s$. In particular, we get that $\g_n x^- \in \L - \set{\s} \subset S$ and hence
            \[ \mathrm{height}_\s(\g_n r(t_m)) \r 0, \]
            as $m\r \infty$ (see~\eqref{defn: height} for the definition of $\mathrm{height}_\s$).
            Thus, it follows that for all $m \gg n$,
            \[ \mathrm{height}_\s(\g_n r(t_m)) \leqslant \mathrm{height}_\s(\g_n r(t_n)). \]           
            However, by~\eqref{eqn: geodesic meets H_n}, we have that $\mrm{height}_\s(\g_n r(t_n)) \r \infty$.
            In particular, we can find some $m>n$ so that the following inequalities hold.
            \begin{equation*}
            \mathrm{height}_\s( \g_m r(t_m)) \geqslant \mathrm{height}_\s(\g_n r(t_n)), \quad 
            \mathrm{height}_\s(\g_n r(t_m)) \leqslant \mathrm{height}_\s(\g_n r(t_n)).
            \end{equation*}
            Now note that $\g_m r(t_m) = (\g_m \g_n^{-1}) \g_n r(t_m)$ while $\g_m \g_n^{-1}$ is a horizontal translation which does not increase height, a contradiction.
            This proves our claim that $\g_n x^+ \neq \s $ for $n\gg 1$.
            
            This argument implies that the geodesic lines $\g_n l(x^-,x^+)$ are not vertical lines and in particular that
            \[ \g_n x^+,\g_n x^- \in \L \cap \overline{S}. \]
            The next step is to show that $\g_n C$ is a Euclidean circle for $n\gg 1$, i.e. not a straight line.
            Assume otherwise.
            Since $\g_n l(x^-, x^+) \cap H_n \neq \emptyset$, we necessarily get that the Euclidean distance between $\g_n x^-$ and $\g_n x^+$ tends to $\infty$.
            But, since both $\g_n x^-$ and $\g_n x^+$ belong to $\overline{S}$, we see that the angle $\th_n$ between the lines $L_1$ and $\g_n C$ must tend to $0$, contrary to our assumption.
            
            Next, we wish to show that $|\g_n C \cap L_i | = 2$ for $i=1,2$ and $n\gg 1$.
            Since $\g_n C$ are Euclidean circles and $\g_n \cap L_i \neq \emptyset$ for $i=1,2$ and $n\gg 1$ by assumption, we may assume that $|\g_n C \cap L_1| =2$.
            If $|\g_n C \cap L_2| =1$ for $n\gg 1$, it is straight forward to check that $\th_n \r 0$ in this case using the fact that $\g_n x^\pm \in \overline{S}$ and the Euclidean distance between $\g_n x^-$ and $\g_n x^+$ tends to $\infty$.

            This completes the verification of the hypotheses of Lemma~\ref{angles lemma}.
            Thus, if we let $\g_n C \cap L_1 = \set{ u_n,v_n}$ and $\g_n C \cap L_2 = \set{w_n,z_n}$, we get that
            \begin{align*}
            	d_{\H^3} (l(u_n,v_n) , l(w_n,z_n)) \r 0.
            \end{align*}
            
            Next, we verify the hypotheses of Lemma~\ref{symmetrization lemma}. 
            The first $4$ conditions have been already verified. As for the last condition, we may left multiply $\g_n$ by an element of $\G_\s$ so as to translate the midpoint between $u_n$ and $v_n$ so that it's inside $F\cap L_1$.
            Since the distance between $\g_n x^-$ and $\g_n x^+$ tends to $\infty$, along with our assumption that $\th_n$ remains bounded away from $0$, it is elementary to check that this forces $|u_n|, |v_n|, |w_n|$ and $|z_n|$ to tend to $\infty$, which is the last condition in the Lemma.            
            Lemma~\ref{symmetrization lemma} then implies that $R(x)$ is not $K$-thick for any $K>1$ which is the desired contradiction.
        \end{proof}
        

\section{A Recurrence-free Criterion for Denseness}

    This section is dedicated to the proof of Theorem~\ref{accumulating on a horocycle implies dense}.
    Throughout this section, we assume $\G$ is a torsion-free geometrically finite Kleinian group with limit set $\L$ and that $\L$ is a circle packing. We use $\mc{C}_\L$ to denote the set of circles which meet $\L$.
    The following lemma is the key ingredient in the proof.

	\begin{lemma} \label{inflate the balloon}
		Let $E$ be a circle such that $\G E$ is closed. Let $C_n$ be a sequence of circles such that $|C_n \cap \L(\G^E)| = 2$ and $C_n$ converge to a circle $D$ which is tangent to $E$. Then, $E \in \overline{\bigcup_n \G C_n}$.
	\end{lemma}
    
    \begin{proof}
		Let $B \subset \mathbb{S}^2$ be the disk bounded by $E$ such that $D\cap B \neq \emptyset$. Since $\G$ is torsion free, we have that $\G^B = \G^E$.
        Hence, we can endow $B$ with a hyperbolic metric in which $\G^E$ acts by isometries.
        Moreover, by Theorem~\ref{limit set of a closed orbit}, we have $\G^E$ is finitely generated. Thus, $\G^E$ acts on $B$ as a geometrically finite Fuchsian group.
        
        Let $p_1,\dots, p_k \in E$ be a maximal set of inequivalent parabolic fixed points under the action of $\G^E$.
        Let $H_1, \dots, H_k \subset B$ be open disjoint horoballs such that $H_i$ is centered at $p_i$ for $i=1,\dots, k$. Since $\G^E$ is geometrically finite, we can also choose these horoballs to satisfy conditions (a), (b) and (c) in \S~\ref{cuspnbhd}.
        Moreover, we get that $ \G^E \backslash (\mathrm{hull}(\L(\G^E))- \cup_i \G^E H_i)$ is compact. Here, we take $\mathrm{hull}(\L(\G^E))$ to be a subset of $B$ in its hyprebolic metric.
        
        Next, for each $n$, let $a_n,b_n \in \L(\G^E)$ be the $2$ distinct points where $C_n$ meets $E$. Let $l_n \subset B$ denote the geodesic joining $a_n$ and $b_n$ in the hyperbolic metric on $B$.
        
        We claim that for all $n$,
        \[  (\mathrm{hull}(\L(\G^E)) - \cup_i \G^E H_i) \cap l_n \neq \emptyset. \]
        
        To see this, fix some $n$ and suppose $l_n \cap \a H_i \neq \emptyset$ for some $\a \in \G^E$ and some $ i$. Then, since $l_n$ must exit $\a H_i$, one has $l_n \cap \partial \a H_i \neq \emptyset$. However, from conditions (b) and (c) in \S~\ref{cuspnbhd}, the horoballs in $\G^E H_i$ have disjoint closures, and thus 
        \[ l_n \cap \partial \a H_i \subset (\mathrm{hull}(\L(\G^E)) - \cup_i \G^E H_i). \]
        
        But, then, by compactness of $ \G^E \backslash (\mathrm{hull}(\L(\G^E))- \cup_i \G^E H_i)$, for each $n$, one can find an element $\b_n \in \G^E$ such that $\b_n l_n$ meets some fixed compact set $K \subset B$.
        
        Pass to a subsequence and assume $\b_n a_n$ and $\b_n b_n$ converge to $a$ and $b$ respectively. Since $\b_nl_n$ meets $K$ for all $n$, we necessarily have that $a\neq b$. Thus, we get that that the sequence of circles $\b_n C_n$ converges to some circle $F$ passing through $a$ and $b$. We claim that $F = E$.
        
        Note that if $F \neq E$, then $F\cap B \neq \emptyset$. Now, let $A_n$ denote the arc $C_n \cap B $ for each $n$. Note that, by assumption, $C_n \r D$ where $D$ is a horocycle in the hyperbolic metric on $B$ (being a circle internally tangent to its boundary). Therefore, we get that the geodesic curvature of $A_n$ in the hyperbolic metric on $B$ tends to $1$ as $n\r \infty$.
        
        However, since $\G^E$ acts by isometries on $B$, the geodesic curvature of $A_n$ and $\b_n A_n$ are equal. Therefore, the geodesic curvature of the arc $F \cap B$ must be equal to $1$. 
        
        But, if $F\cap B$ is an arc of a circle meeting $E$ in two points, which is an equidistant curve from the geodesic joining $a$ and $b$, its curvature cannot be equal to $1$.
        Thus, we get that $F = E$ as desired.
        
	\end{proof}


    Theorem~\ref{accumulating on a horocycle implies dense} will now follow from Lemma~\ref{inflate the balloon} and Theorem~\ref{accumulating on the boundary of a component}~\cite[Theorem 5.1]{MMO-Planes}.
    \begin{proof} [Proof of Theorem~\ref{accumulating on a horocycle implies dense}]
    
    	In the notation of Theorem~\ref{accumulating on a horocycle implies dense}, let $\s$ be the point where $D$ meets $\L$. Then there exists a component $B\subset \Omega$ such that
    \[ D-\set{\s} \subset B. \]
    	Then, since $\G$ is assumed to be geoemtrically finite, we have that $\L$ is not contained in a circle and thus, 
        \[ (\mathbb{S}^2 - \overline{B}) \cap \L \neq \emptyset. \]
        
        Let $\g_n\in \G$ be a sequence such that $\g_n C \r D$. Since $|C\cap \L| > 1$, we have that for all $n$:
        \[ \g_n C \cap  (\mathbb{S}^2 - \overline{B}) \neq \emptyset. \]
        
		Let $E = \partial B$. We have that $\L(\G^E) = E$, since $\G^E\backslash B$ is a finite volume surface. Thus, we have that
        \[ |\g_n C \cap E| = |\g_n C \cap \L(\G^E)| = 2. \]
        
        Now, we can apply Lemma~\ref{inflate the balloon} with $C_n = \g_n C$, $D$ and $E$ to conclude that, after passing to a subsequence if necessary, $\g_n C \r E$.
        Therefore, we can apply Theorem~\ref{accumulating on the boundary of a component} to conclude that
        \[ \overline{\G C} = \mc{C}_\L. \]
	\end{proof}
    
    Theorem~\ref{accumulating on a horocycle implies dense} has the following corollary: if a plane $P$ meets the convex core of $M$ and such that $P$ is not closed in $M$, then the closure of $P$ contains uncountably many planes $Q$ whose invariant circle meets $\L$ in at most $2$ parabolic points.
    The precise statement is the following:
    
    \begin{corollary} \label{non-closed orbits contain uncountably many circles with radial points}
    	Let $C$ be a circle such that $|C\cap \L| > 1$ and $\G C$ is not closed. Then, $\overline{\G C}$ contains uncountably many circles, each meeting $\L$ in uncountably many points, at most $2$ of which are parabolic fixed points.
    \end{corollary}
    
    \begin{proof}
    	If $\G C$ is dense in $\mc{C}_\L$, then the statement follows trivially.
        If not, then by Theorem~\ref{accumulating on a horocycle implies dense}, we know that $\overline{\G C}$ cannot contain a circle meeting $\L$ in only one point.
        Moreover, by homogeneity, we know that $\overline{\G C}$ contains uncountably many circles.
        On the other hand, since any $3$ points determine a unique circle and there are only countably many parabolic fixed points in $\L$, we see that $\overline{\G C}$ contains at most countably many circles meeting $\L$ in $3$ or more parabolic points.
        Thus, the claim follows.
    \end{proof}
  
\section{Elementary Closed Orbits}
\label{section: elementary orbits}	
    In this section, we prove sufficient conditions for orbits to be closed. Throughout the section, we fix the following notation
    
    \begin{itemize}
      \item $\G$ is a torsion-free geometrically finite Kleinian group.
      \item The limit set $\L$ is a circle packing i.e. $\Omega = \mathbb{S}^2 - \L$ is a union of round open disks.
      \item $C\subset \mathbb{S}^2$ is a circle.
    \end{itemize}

\subsection{Orbits of Parabolic Subgroups}
	
    This section is dedicated to studying orbits of circles under the action of parabolic subgroups of $\G$.
    A useful first observation is the following
	
    \begin{remark} \label{discrete iff closed}
		The orbit $\G C$ is closed if and only if it is discrete. This follows from the homogeneity of the orbit and countability of $\G$.
	\end{remark}
    
    \begin{lemma} \label{rk1 parabolic orbits are closed}
		Let $\s \in \L$ be a rank $1$ parabolic fixed point and let $\G_\s$ be its stabilizer. Then, $\G_\s C$ is closed. 
	\end{lemma}
    
    \begin{proof}
		Choose coordinates in the upper half space model of $\H^3$ such that $\s = \infty$. In those coordinates, $\G_\s\cong \Z$ acts by discrete translations on the complex plane, fixing $\infty$. Thus, $\G_\s C$ is a discrete set of parallel circles (or lines if $\s \in C$) in the plane. Hence, $\G_\s C$ is closed, by remark~\ref{discrete iff closed}.
	\end{proof}
    
    \begin{definition} \label{rational slope}
		Let $\s \in \L$ be a rank $2$ parabolic fixed point and let $\G_\s$ be its stabilizer. Put $\s = \infty$ in the upper half space model and let $u$ and $v$ be
        a basis for the action of $\G_\s$ on $\R^2$ i.e. $u$ and $v$ define the $2$ lines parallel to the axes of translations of $\G_\s$.
        A circle $C$ is said to have a $\s$-\textbf{\textit{rational slope}} if $\s \in C$,
        and for any vector $w = au+bv \in \R^2$ parallel to $C$, $b/a \in \Q$.
	\end{definition}

    \begin{lemma} \label{orbit under rk2 parabolic group}
		Let $C$ be a circle such that $C\cap \L$ contains a rank $2$ parabolic fixed point $\s$. Then, $\G_\s C$ is closed if and only if $C$ has a $\s$-rational slope.
	\end{lemma}
    
    \begin{proof}
    	This is an exact analogue of the dichotomy between the behavior of lines with rational and irrational slopes on the standard torus.
    \end{proof}
    
    \begin{remark} \label{remark: rk2 cusps and density}
    It should be noted that that in Lemma~\ref{orbit under rk2 parabolic group}, if $\G_\s C$ is not closed, then $\G C$ is in fact dense in the space of circles which meet $\L$.
    \end{remark}
    
    \subsection{Elementary Surfaces}
    In the remainder of this section, we study planes whose circle at infinity meets $\L$ in finitely many (parabolic) points.
    We prove that all such plane immersions are closed which is the content of Theorem~\ref{finite intersection implies closed}.
    
    We split this study into $3$ cases depending on the cardinality of the intersection of a circle with the limit set.
    The proof relies on endowing connected components of the domain of discontinuity with a hyperbolic metric in which the $\G$ action is by isometries.
    This technique was used fruitfully in~\cite{MMO-Planes}.
    
	\begin{lemma} \label{emptyclosed}
		Let $C\subset \mathbb{S}^2$ be a circle such that $C\cap \L = \emptyset$, then $\G C$ is closed.
	\end{lemma}
    
    \begin{proof}
		This follows from the fact that the action of $\G$ on $\Omega$ is properly discontinuous.
	\end{proof}

    \begin{lemma}[cf. Section $3$ of~\cite{MMO-Planes}] \label{onepointclosed}
		Let $C$ be a circle meeting $\L$ in only one point, $\s$. Suppose $\s$ is a parabolic fixed point. Then, $\G C$ is closed.
	\end{lemma}
    
    \begin{proof}
		Let $B$ be the component of the domain of discontinuity $\Omega$ such that $C-\set{\s} \subset B$, so that $\s\in \partial B$. Let $(\g_n)$ be a sequence of elements of $\G$. We shall show that
        \[ \overline{\set{\g_n C}} \subseteq \G C. \] 
        
        If for infinitely many $n\neq m$, we have that $\g_n B \neq \g_m B$, then, after passing to a subsequence, we get $area(\g_n B) \r 0$. But, then we get that the diameter of $C$ also tends to $0$.
        
		Thus, after passing to a subsequence, we may assume that there exists some $n_0 \geq 1$ such that $\g_n B = \g_{n_0} B$ for all $n \geq n_0$. So, we can reduce to the case where
        \[ \g_n \in \G^{B} \]
        for all $n$, where $\G^B$ denotes the stabilizer of $B$ in $\G$.
        
        Note that $B$, being a round disk, can be endowed with the hyperbolic metric on $\H^2$. In this metric, the action $\G^B$ is by isometries and $C$ is a horocycle. Also, $\s$ is a parabolic fixed point for $\G^B$.
        Thus, the orbit of $C$ under $\G^B$ is closed by~\cite[Proposition C]{Dalbo2000}.
	\end{proof}
    

	\subsection*{Proof of Theorem~\ref{finite intersection implies closed}}
    	
    	By Lemmas~\ref{emptyclosed} and~\ref{onepointclosed}, it remains to show the conclusion when $1 < |C\cap \L| < \infty$.
		Let $\s \in C\cap \L$. Since $\s$ is isolated in $C\cap \L$ and $|C\cap \L| >1$, there exist $2$ distinct components $B_1, B_2 \subset \Omega$ such that $\partial B_1 \cap \partial B_2 = \set{\s}$. We say $\s$ is a point of tangency of $B_1$ and $B_2$.
        
        Let $B_1,\dots, B_n$ be the distinct components of $\Omega$ which $C$ meets. Let $C_i = C\cap \overline{B}_i$, for $i=1,\dots, n$. Then, each $C_i$ is an arc of positive length and $C = \cup_i C_i$. Moreover, one has        
        \[ \mathrm{length}(C) = \sum_{i=1}^n \mathrm{length}(C_i), \]
        where $\mathrm{length}(\cdot)$ is taken in the Euclidean metric on $\mathbb{S}^2$. Note that $\mathrm{length}(C_i)$  is bounded above by the circumference of $B_i$, for each $i$.

        Let $(\g_k)$ be a sequence of elements of $\G$. We shall show that
        \[ \overline{\set{\g_k C}} \subseteq \G C. \] 
        
        Suppose that for infinitely many $k\neq j$, we have that $\g_k B_i \neq \g_j B_i$, for some fixed $i \in \set{1,\dots,n}$. Then, after passing to a subsequence, we get $area(\g_k B_i) \r 0$. But, then we get $\mathrm{length}(C_i)$ also tends to $0$. If this happens for all $B_i$, then $\mathrm{length}(C) \r 0$ as $k\r \infty$.
        
        Thus, after passing to a subsequence, we may assume without loss of generality that for all $k\geq 1$
        \[ \g_k B_1 = B_1. \]
        In particular, $\g_k \in \G^{B_1}$. Write $C\cap \partial B_1 = \set{\s_1,\s_2}$. As in the proof of Lemma~\ref{onepointclosed}, we endow $B_1$ with the hyperbolic metric on $\H^2$ in which $\G^{B_1}$ acts by isometries. Denote this metric by $d_{\H^2}$. Let $l$ denote the geodesic in this metric whose endpoints are $\s_1$ and $\s_2$.
        
        Now, note that $C_1$ is an equidistant curve from $l$, i.e. there exists a constant $r>0$, for all $x\in C_1 - \set{\s_1,\s_2}$ such that
        \[ d_{\H^2}(x,l) < r. \]
        
        Thus, $\g_k C_1$ is equidistant from $\g_k l$ for all $k$.
        Now, suppose that for all $i \neq 1$, we have that $area(\g_k B_i) \r 0$. Then, it must be that $\g_k \s_1 \r \s$ and $\g_k \s_2 \r \s$ for the same point $\s \in \partial B_1$. But, since $\g_k l$ is the geodesic connecting $\g_k \s_1$ and $\g_k \s_2$, we have that the Euclidean length of $\g_k l$ must tend to $0$ as $k \r \infty$.
        
        Therefore, we get that the Euclidean length of $\g_k C_1$  also tends to $0$, being equidistant curves from $\g_k l$ (in $d_{\H^2}$). 
        Finally, since $area(\g_k B_i) \r 0$ for all $i\neq 1$ by assumption, we get that $\mathrm{length}(\g_k C_i) \r  0$ for all $i\neq 1$. But, then $\mathrm{length}(\g_k C) \r 0$ as $k \r\infty$. And, thus, there exists some $i\neq 1$ such that $area(\g_k B_i) \nrightarrow 0$.
        
        Now, assume that $B_1,\dots,B_n$ are ordered using the cyclic order defined by the circle $C$. Then, choose $B_j$ ($j\neq 1$) so that all the components $B_i$ for $i$ between $1$ and $j$ (in the cyclic order on $C$) on one side have that $area(\g_kB_i) \r 0$.
        
        Thus, after passing to a further subsequence, since $area(\g_kB_j) \nrightarrow 0$, then for all $k \geq 1$,
        \[ \g_k B_j = B_j. \]
        
        Hence, we get that for all $k\geq 1$
        \[ \g_k \in \G^{B_1} \cap \G^{B_j}. \]
        
        Next, suppose that $C \cap \partial B_1 \cap \partial B_j \neq \emptyset$. Then, without loss of generality, $\s_1 \in C \cap \partial B_1 \cap \partial B_j$.
        Thus, we get that $\G^{B_1} \cap \G^{B_j} \subseteq \G_{\s_1}$ and, so, for all $k\geq 1$
        \[ \g_k \in \G_{\s_1}. \]
        
        Hence, since $\s_1$ is a rank $1$ parabolic fixed point, by Lemma~\ref{rk1 parabolic orbits are closed}, we get that
        \[ \overline{\set{\g_k C}} \subseteq \overline{\G_{\s_1} C} = \G_{\s_1} C \subseteq  \G C,  \]
        as desired.
        
        Otherwise, write $C\cap \partial B_2 = \set{\s_3,\s_4}$. Then, since the length of one of the arcs of $C$ connecting $B_1$ and $B_j$, say the arc connecting $\s_1$ and $\s_3$, goes to $0$ under $\g_k$, there exists $\a \in \partial B_1 \cap \partial B_j $ such that $\g_k \s_1 \r \a$ and $\g_k \s_3 \r \a$. But, then we get that
        \[ \G^{B_1} \cap \G^{B_k} \subseteq \G_\a. \]
        And, thus, similarly, since $\a$ is a rank $1$ parabolic fixed point (being a tangency point of two components of $\Omega$), we get the desired conclusion.
	

\section{Limit Sets of Closed Orbits}
    \label{section: closed orbits}

    In this section, we prove Theorem~\ref{limit set of a closed orbit} concerning the rigidity of closed plane immersions. The proof is broken up into separate cases handled by the following Lemmas. A complete proof of the theorem is given at the end of the section.
    Recall that for a circle $C$, $\G^C$ denotes the stabilizer of $C$ in $\G$.
    For a point $\s \in \partial \H^3$, $\G_\s$ denotes the stabilizer of $\s$ in $\G$.
    
    \begin{lemma} \label{limit set of elementary orbits}
		Let $C$ be a circle such that $\G C$ is closed and $1<|C\cap \L| < \infty$. Then, $\L(\G^C) = \emptyset$.
	\end{lemma}
    
    \begin{proof}
    Observe first that $\L(\G^C) \subseteq \L\cap C$.
    In particular, the finiteness of the limit set of the group $\G^C$ implies that it is elementary (i.e. virtually abelian).
    Hence, each point of $\L(\G^C)$ is fixed by an element of $\G^C$.
    
		Suppose towards a contradiction that $\L(\G^C) \neq \emptyset$. Let $\s \in \L(\G^C)$. Choose coordinates in the upper half space model so that $\s = \infty$ and $C$ is the real axis. If $\s$ is a parabolic point, then $\G_\s \cap \G^C$ acts by horizontal translations. Note that because $\L$ is a circle packing, $|C\cap \L| \geq 3$. Thus, there exists a point $\s_2 \in C\cap \L$, such that $\s \neq \s_2$. And, hence, we get
        \[ (\G_\s\cap\G^C)\cdot \s_2 \subseteq C\cap \L. \]
        
        But, $(\G_\s\cap\G^C)\cdot \s_2$ is an infinite collection of points, contrary to the assumption that $|C\cap \L| < \infty$. Thus, $\s$ must be fixed by a hyperbolic isometry $\g \in \G^C$. Let $\s_2\in C\cap \L$ be the other fixed point of $\g$ and without loss of generality, assume $\s$ is the attracting fixed point.
        
        Choose coordinates so that $\s = \infty$, $\s_2 = 0$ and $C$ is the real axis in the upper half space model. In these coordinates, $\g$ acts on $C$ by $x\mapsto xt$, for all $x\in \R$ and some $t>0$. Since $|C\cap \L| \geq 3$, we can find $\s_3$ in $C\cap \L$, different from $\s$ and $\s_2$. But, then, we get that $\g^n \s_3 \in C\cap \L$ for all $n\in \Z$. This again contradicts the finiteness of $|C\cap \L|$. 
	\end{proof}
    
    
    \begin{lemma} \label{nonisolated parabolic pt}
		Let $C$ be such that $\G C$ is closed. Let $\s$ be a non-isolated point in $C\cap \L$ and assume $\s$ is a parabolic fixed point. Then, $\s \in \L(\G^C)$.
	\end{lemma}
    
    \begin{proof}
		Choose coordinates in the upper half space model so that $\s = \infty$ and $C$ is the real axis. Then, $\G_\s$ acts by translations on the complex plane. If $\s$ is a rank $1$ parabolic fixed point, then $\L-\set{\s}$ lies inside an infinite strip between two parallel straight lines $L_1, L_2$ and $\G_\s$ acts by translations in the direction parallel to these two lines.
        
        Since $\s$ is not isolated in $C\cap \L$, then $C \cap \L \subseteq \R$ contains sequence $\l_n \r \pm \infty$. Hence, $C$ meets the strip containing the limit set infinitely often and thus must be parallel to $L_1$ and $L_2$. Hence, $C$ is invariant under $\G_\s$. That is $\G_\s \subseteq \G^C$ and so $\s \in \L(\G^C)$.
        
        If $\s$ is a rank $2$ parabolic point, then by Lemma~\ref{orbit under rk2 parabolic group} and Remark~\ref{remark: rk2 cusps and density}, since $\G C$ is closed, $C$ must have a $\s$-rational slope. Moreover, by the proof of Lemma~\ref{orbit under rk2 parabolic group}, $C$ is invariant under an infinite cyclic subgroup of $\G_\s$. Thus, in particular, $\s$ is fixed by a non-trivial element of $\G^C$ and so $\s \in \L(\G^C)$.
	\end{proof}
    
    \begin{remark}
		Notice that the assumption that $\G C$ is closed in the statement of the above lemma was only used when $\s$ was a rank $2$ parabolic fixed point.
	\end{remark}
    

    \begin{lemma} \label{nonisolated radial point}
		Let $C$ be such that $\G C$ is closed. Let $\s$ be a non-isolated point in $C\cap \L$ and assume $\s$ is a radial limit point. Then, $\s \in \L(\G^C)$.
	\end{lemma}
    
    \begin{proof}
		Let $o \in \mathrm{hull}(C)$. Since $\s\in \L$, there exists a sequence $(\g_n)$ of elements in $\G$ such that $\g_n o \r \s$. Moreover, since $\s$ is a radial limit point, there exists a constant $r>0$, such that for all $n\geq 1$:
        \[ d_{\H^3}(\g_n o, l(o,\s)) = d_{\H^3}( o, \g_n^{-1} l(o,\s)) < r, \]
        where $l(o,\s)$ is the geodesic ray starting from $o$ and ending at $\s$. Note that $l(o,\s) \subset \mathrm{hull}(C)$.
        Thus, we get that for all $n \geq 1$,
        	\[  \mathrm{hull}(\g_n^{-1} C) \cap \overline{B(o, r)} \neq \emptyset.  \]
            
        But, if the collection of circles $\g_n^{-1}C$ is infinite, then, by Proposition~\ref{circles meeting cpt set}, the circles $\g_n^{-1}C$ must accumulate. But, this would contradict the discreteness of the orbit $\G C$. Hence, the collection of circles $\g_n^{-1}C$ must be finite.        
      Thus, after passing to a subsequence and without loss of generality, we may assume that for all $n\geq 1$,
      \[ \g_n \in \G^C. \]
      But, then $\g_n o \in \mathrm{hull}(C)$ for all $n$ and hence $\s \in \L(\G^C)$ as desired.
	\end{proof}
    
    \begin{remark} 
		If $\s \in C\cap \L$ is a radial limit point which is isolated in $C\cap \L$, then $C\cap \L = \set{\s}$ and $\G C$ is not closed (by Lemma~\ref{onepointclosed}). Hence, Lemma~\ref{nonisolated radial point} says that if $ |C\cap \L| > 1$ and $\G C$ is closed, then \textit{every} radial limit point in $C\cap \L$ belongs to $\L(\G^C)$.
	\end{remark}
    

    \begin{proof}[Proof of Theorem~\ref{limit set of a closed orbit}]
		First, we automatically have $\L(\G^C) \subseteq C\cap \L$. Assume that $|C\cap \L| =1$. Then, by Lemma~\ref{onepointclosed} (and its proof), $C$ meets $\L$ in a rank $1$ parabolic fixed point $\s$ and $C$ is invariant under $\G_\s$. Thus, $\G_\s \subseteq \G^C$ and hence $\s \in \L(\G^C)$.
        
        Next, assume that $1<|C\cap \L| <\infty $. Then, by Lemma~\ref{limit set of elementary orbits}, we have that $\L(\G^C) = \emptyset$ and the claim follows in this case.
        
        Now, assume $|C\cap \L| = \infty$ and that $\G C$ is closed. Define the following set
        \[ \mc{A}(C) = \set{\l\in C\cap\L: \exists \l_n \in C\cap \L, \l_n \neq \l, \l_n \r \l }. \]
        
        We need to show that $\L(\G^C) = \mc{A}(C)$. Note that $\L(\G^C) = 0,1,2,$ or $\infty$. First, assume that $\G^C$ is elementary, so that $|\L(\G^C)| = 0,1$, or $2$. Then, the same argument used in the proof of Lemma~\ref{limit set of elementary orbits} shows that in those cases, we always get that $\L(\G^C) \subseteq \mc{A}(C)$. On the other hand, if $\L(\G^C)$ is infinite, then it is known that $\L(\G^C)$ is a perfect closed set with no isolated points yielding $\L(\G^C) \subseteq \mc{A}(C)$ in this case as well.
        
        For the reverse inclusion, let $\s \in \mc{A}(C)$. Since $\G$ is geometrically finite, $\s$ is either a parabolic or a radial limit point. In either case, Lemmas~\ref{nonisolated parabolic pt} and~\ref{nonisolated radial point} give that $\s \in \L(\G^C)$ as desired.
        
        Moreover, the proof of Lemma~\ref{nonisolated parabolic pt} (resp. Lemma~\ref{nonisolated radial point}) show that if $\s \in \mc{A}(C)$ is parabolic (resp. radial) for $\G$, then $\s$ is also parabolic (resp. radial) for $\G^C$. Again, since $\G$ is geometrically finite,
        any point in $\mc{A}(C)$ is either radial or parabolic. Hence, $\L(\G^C)$ consists of parabolic and radial limit points which proves that $\G^C$ is a (non-elementary) geometrically finite\footnote{For fuchsian groups, every parabolic point is bounded.} Fuchsian group in this case. Hence, $\G^C$ is finitely generated.
	\end{proof}


\section{Isolation and Finiteness of Elementary Orbits}
\label{section: discreteness elementary}

	The purpose of this section is to prove a weak form of isolation for closed $H$-orbits with elementary stabilizers meeting $Core(M)$; Proposition~\ref{finiteness of orbits in each degree}.
    This result follows immediately from the following lemma.
	
	\begin{lemma} \label{discreteness of B_k}
		Let $k\in \N$ be such that $k \geq 3$ and let $\mc{B}_k$ denote the set of circles $C\subset \mathbb{S}^2$ such that $|C\cap \L| = k$. Then, $\mc{B}_k$ is discrete in the space of circles.
		Moreover, if a circle $C$ with $|C\cap \L|<\infty$ is an accumulation of circles in $\mc{B}_k$, then $|C\cap \L|<k$.
	\end{lemma}
    
    \begin{proof} 
    	
		Suppose that there is some $k \geq 3$ and $C \in \mc{B}_k$ such that there exists a sequence of circles $C_n \in \mc{B}_k$ converging to $C$.
        
        Let $B_1,\dots,B_k$ be the components of $\Omega$ meeting $C$. Let $\eta$ be the center of one of the disks in $\mathbb{S}^2$ bounded by $C$ and let $r$ be its radius. Let $\e >0$ be small enough so that the $2$ circles bounding the $2$ disks centered around $\eta$ and of radius $r-\e$ and $r+\e$ respectively meet $B_i$ for all $i$. This is possible because we only have finitely many components $B_i$.
        
        Next, let $\mc{N}_\e(C)$ be the $\e$-annulus around $C$ (Definition~\ref{annulus}). Then, by Proposition~\ref{annuli are open}, we have that for all $n\gg 1$,
        	\[C_n \in \mc{N}_\e(C).\]
        But, one has that any circle lying entirely inside the annulus $\mc{N}_\e(C)$ must meet $B_i$, for all $i$ by choice of $\e$. Moreover, any circle lying inside $\mc{N}_\e(C)$ and meeting exactly $k$ components of $\Omega$ must pass through all $k$ tangency points of these components. Since $k\geq 3$, there is one unique such circle which is $C$. Thus, $C_n$ must meet strictly more than $k$ components of $\Omega$ for all $n\gg 1$, which contradicts the fact that $C_n \in \mc{B}_k$ for all $n$. The same argument also implies the second assertion of the lemma.
	\end{proof}

\section*{Acknowledgements}
I would like to thank my advisor, Nimish Shah, for suggesting the problem and for his guidance, support and for numerous valuable discussions. I thank the anonymous referees for carefully reading this article and the suggestions that improved the exposition.
I would like to thank Yongquan Zhang for pointing out an error in Proposition~\ref{finiteness of orbits in each degree} in a previous version.

\newcommand{\etalchar}[1]{$^{#1}$}

\bigskip

\end{document}